\documentclass[12pt]{article}

\usepackage{amsmath,amsfonts,amssymb,amsthm,cite}
\usepackage{geometry}

\theoremstyle{plain}
\newtheorem{theorem}{Theorem}
\newtheorem{lemma}{Lemma}

\theoremstyle{definition}

\newcommand{\nonprint}[1]{}

\begin{document}

\begin{flushleft}

\small
DOI 10.1007/s11253-019-01588-w\\
\textit{Ukrainian Mathematical Journal, Vol.70, No.10, March, 2019 (Ukrainian Original Vol.70, No.10, October, 2018)}
\vspace{+0.2cm}

\medskip

\normalsize

\textbf{O.\,M.~Atlasiuk, V.\,A.~Mikhailets} \small(Institute of Mathematics of NAS of Ukraine, Kyiv)

\Large

\textbf{Fredholm one-dimensional boundary-value problems\\ in Sobolev spaces}
\end{flushleft}

\normalsize

\begin{abstract}
For systems of ordinary differential equations on a compact interval, we study the character of solvability
of the most general linear boundary-value problems in Sobolev spaces. We find the indices of these
problems and obtain a criterion of their well-posedness.
\end{abstract}

\section{Introduction}\label{section1}

The investigation of the solutions of systems of ordinary differential equations is an important part of numerous
problems of contemporary analysis and its applications (see, e.g.,~\cite{BochSAM2004}~and the references therein). For general linear boundary-value problems, the conditions required for the Fredholm property and the continuous dependence of the
solutions on parameters were established by Kiguradze \cite{Kigyradze1975, Kigyradze1987}. Later, the accumulated results were developed by the second author of the present paper and his colleagues \cite{KodliukMikhailets2013JMS, MPR2018, MikhailetsChekhanova}. Recently, these investigations were extended to more general classes of Fredholm boundary-value problems in various Banach function spaces \cite{GKM2015, KodlyukM2013, GKM2017, MMS2016}. These problems have numerous specific features that are not typical of ordinary boundary-value problems and require the use of new approaches and methods. In the~present paper, we develop these approaches and methods.

\section{Statement of the problem}\label{section2}

Consider a finite segment $[a,b]\subset\mathbb{R}$ and given parameters
$$
\{m, n, r\} \subset \mathbb{N}, \quad 1\leqslant p\leqslant \infty.
$$
By $W_p^n:=W_p^n([a,b];\mathbb{C})$ we denote a complex Sobolev space and set $W_p^{0}:=L_p$. Similarly, by $(W_p^n)^{m}:=W_p^n([a,b];\mathbb{C}^{m})$ and $(W_p^n)^{m\times m}:=W_p^n([a,b];\mathbb{C}^{m\times m})$ we denote Sobolev spaces of vector functions and matrix functions, respectively, whose elements belong to the function space $W_p^n$. By $\|\cdot\|_{n, p}$ we denote the norms in these spaces. They are defined as the sums of~the corresponding norms of all elements of a vector-valued or matrix-valued function in $W_p^n$. The~space of functions (scalar functions, vector functions, or matrix functions) in which
the norm is introduced is always clear from the context. For $m=1$, all these spaces coincide. It is known that the spaces $W_p^n$ are Banach spaces. They are separable if and only if $p<\infty$.

Consider a linear boundary-value problem for the system of $m$ differentiable equations of the first order
\begin{equation}\label{1.6.1t}
    Ly(t):= y'(t)+A(t)y(t)=f(t),\quad
t \in(a,b),
\end{equation}
\begin{equation}\label{1.6.2t}
    By= c,
\end{equation}
\noindent where the matrix function $A(\cdot)$ belongs to the space $(W_{p}^{n-1})^{m\times m}$, the vector function $f(\cdot)$ belongs to the space $(W_{p}^{n-1})^{m}$, the vector $c$ belongs to the space $\mathbb{C}^{r}$, and~$B$ is a linear continuous operator
\begin{equation}\label{oper B}
  B\colon \ (W_{p}^{n})^{m} \rightarrow\mathbb{C}^{r}.
\end{equation}

We represent vectors and vector functions in the form of columns. A solution of the~boundary-value problem \eqref{1.6.1t}, \eqref{1.6.2t} is understood as a vector function $y(\cdot)\in (W_{p}^{n})^m$ satisfy\-ing equation~$(\ref{1.6.1t})$ almost everywhere on $(a,b)$ (everywhere for $n\geq 1$) and equality~(\ref{1.6.2t}) specifying $r$ scalar boundary conditions. The solutions of equation \eqref{1.6.1t} fill the space $(W_{p}^{n})^m$ if its right-hand side  $f(\cdot)$ runs through the space $(W_{p}^{n-1})^m$. This statement follows from Lemma \ref{3.sk} (see Section~\ref{section4}). Hence, the boundary condition \eqref{1.6.2t} is the most general condition for this equation and includes all known types of classical boundary conditions, namely, the Cauchy problem, two- and many-point problems, integral and mixed problems, and numerous nonclassical problems. The last class of problems may contain derivatives of the required functions of order $k\leqslant n$.

It follows from the known results of functional analysis \cite{Ioffe} that, for $1\leqslant p < \infty$, every operator~$B$ in \eqref{oper B} admits a single-valued analytic representation
\begin{equation}\label{st anal}
By=\sum _{k=0}^{n-1} \alpha_{k} y^{(k)}(a)+\int_{a}^b \Phi(t)y^{(n)}(t){\rm d}t, \quad y(\cdot)\in (W_{p}^{n})^{m},
\end{equation}
where the matrices $\alpha_{k}$ belong to $\mathbb{C}^{r\times m}$ and the matrix function $\Phi(\cdot)$ belongs to $L_{p^{'}}\big([a, b]; \mathbb{C}^{r\times m}\big)$, $$\frac{1}{p} + \frac{1}{p^{'}}=1.$$ For $p=\infty$, relation \eqref{st anal} also defines an operator $B\in L((W_{\infty}^{n})^{m}; \mathbb{C}^{r})$. However, there exist other operators from this class specified by the integrals over finitely additive measures \cite{Dunford}.

The main aim of the present paper is to prove the Fredholm property for problem \eqref{1.6.1t},~\eqref{1.6.2t} and to find its index. Moreover, we establish a criterion for the (everywhere) single-valued solvability of this problem.

\section{Main results}\label{section3}

We now formulate the main results of the present paper. They are proved in Section~5.

We rewrite the inhomogeneous boundary-value problem \eqref{1.6.1t}, \eqref{1.6.2t} in the form of a linear operator equation
\[ (L,B)y=(f,c), \]
where $(L,B)$ is a linear operator in the pair of Banach spaces
\begin{equation}\label{(L,B)}
(L,B)\colon (W^{n}_p)^m\rightarrow (W^{n-1}_p)^m\times\mathbb{C}^r.
\end{equation}
Recall that a linear continuous operator $T\colon X \rightarrow Y$, where $X$ and $Y$ are Banach operators, is~called a Fredholm operator if its kernel $\ker T$ and cokernel $Y/T(X)$ are finite-dimensional. If~this operator is Fredholm, then its range $T(X)$ is closed in $Y$ and the index
$$
\mathrm{ind}\,T:=\dim\ker T-\dim\big(Y/T(X)\big)
$$
is finite (see, e.g., \cite[Lemma~19.1.1]{Hermander1985}).

\begin{theorem}\label{th_fredh-bis}
The linear operator \eqref{(L,B)} is a bounded Fredholm operator with index $m-r$.
\end{theorem}

We formulate a criterion for the invertibility of the operator $(L,B)$, i.e., the condition under which the inhomogeneous boundary-value problem \eqref{1.6.1t}, \eqref{1.6.2t} possesses a unique solution and this solution continuously depends on the right-hand sides of the differential equation and the~boundary condition.

By $Y(\cdot)\in (W_p^n)^{m\times m}$ we denote a unique solution of a linear homogenous matrix equation of the form \eqref{1.6.1t} with the following Cauchy initial condition:{\samepage
\begin{equation}\label{r3}
   Y'(t)+A(t)Y(t)=0,\quad t\in (a,b), \quad Y(a)=I_{m},
  \end{equation}
\noindent where $I_m$ is the $m \times m$ identity matrix.}

In the case where $r=m$, we set
\begin{equation}\label{3.BY}
[BY]:=\left( B \begin{pmatrix}
                                              y_{1,1}(\cdot) \\
                                              \vdots \\
                                              y_{m,1}(\cdot) \\
                                            \end{pmatrix}
\ldots
                                    B \begin{pmatrix}
                                              y_{1,m}(\cdot) \\
                                              \vdots \\
                                              y_{m,m}(\cdot) \\
                                            \end{pmatrix}\right).
\end{equation}
The numerical square matrix $[BY]$ of order $m$ is formed as a result of the action of the operator~$B$ upon the corresponding
columns (with the same numbers) of the matricant $Y(\cdot)$ of the matrix Cauchy problem~\eqref{r3}.

\begin{theorem}\label{th_invertible-bis}
The operator $(L,B)$ is invertible if and only if $r=m$ and the matrix $\left[BY\right]$ is nondegenerate.
\end{theorem}

\section{Auxiliary results}\label{section4}

We now establish several auxiliary statements and use them to prove Theorems \ref{th_fredh-bis}~and~\ref{th_invertible-bis} in Section~\ref{section5}. Some of
them are of independent interest.

\begin{lemma}\label{3.sk}
Suppose that the matrix function $A$ belongs to $(W_{p}^{n-1})^{m\times m}$. If a differentiable function $y\colon [a,b]\to\mathbb{C}^{m}$ is a solution of equation \eqref{1.6.1t} for some right-hand side $f\in(W_{p}^{n-1})^{m}$, then $y$ belongs to $\in(W_{p}^{n})^{m}$. Moreover, if $f$ runs through the entire space
$(W_{p}^{n-1})^{m}$, then solutions of equation~\eqref{1.6.1t} run through the entire space $(W_{p}^{n})^{m}$.
\end{lemma}

\begin{proof} Assume that, for some $f\in (W_{p}^{n-1})^{m}$, a differentiable vector function $y$ is a solution of equation~\eqref{1.6.1t}. We prove that $y$ belongs to $\in(W_{p}^{n})^{m}$. In view of the fact that $A$ and $f$ are at least continuous on $[a,b]$, we find
$$y'=f-Ay\in\bigl(C^{(0)}\bigr)^{m}.$$
This implies that $y$ belongs to $\bigl(C^{(1)}\bigr)^{m}\subset(L_{p})^{m}$. Moreover,
\begin{equation}\label{3.impl}
\bigl(\,y'\in(W_{p}^{n-1})^{m}\;\Rightarrow\;
y\in(W_{p}^{n})^{m}\,\bigr).
\end{equation}
Indeed, if $y'$ belongs to $(W_{p}^{n-1})^{m}$ for an integer number $n$, then
$$
y'=f-Ay\in(W_{p}^{n-1})^{m}.
$$
Hence, $y$ belongs to $(W_{p}^{n})^{m}$. The inclusion $y\in(L_{p})^{m}$ and property \eqref{3.impl} yield the required inclusion $y\in(W_{p}^{n})^{m}$.

We now prove the last assertion of the lemma. For any $f\in(W_{p}^{n-1})^{m}$, a solution $y$ of equation~\eqref{1.6.1t} exists. As shown above, $y$ belongs to $(W_{p}^{n})^{m}$. In view of the obvious implication
$$
y\in(W_{p}^{n})^{m}\;\Rightarrow\;Ly\in(W_{p}^{n-1})^{m},
$$
this proves the last assertion of the lemma.\end{proof}

\begin{lemma}\label{l1} Suppose that the matrix function $Y(\cdot)\in (W_p^n)^{m\times m}$ is nondegenerate for every $t\in [a, b]$. Then
the inverse matrix function $Y^{-1}(\cdot)$ belongs to $(W_p^n)^{m\times m}$.
\end{lemma}

\begin{proof} We first prove the lemma in the scalar case $m=1$ by induction on $n\in\mathbb{N}$.

Let $n=1$. By the condition, the function $Y(\cdot)$ belongs to $W_p^1$. Hence, it is absolutely continuous and unequal to zero on the set $[a, b]$.
This implies that the function $Y^{-1}(\cdot)$ is differentiable almost everywhere and, in addition, $$\left(Y^{-1}\right)'(\cdot)=-Y'(\cdot)Y^{-2}(\cdot).$$
Since the function $Y(\cdot)$ is separate from zero and $Y'(\cdot)$ belongs to $L_p$, the function $\left(Y^{-1}\right)'(\cdot)$ belongs to $L_p$. Hence, $Y(\cdot)^{-1}$ belongs to $W_p^1$.

Assume that the assertion of Lemma \ref{l1} is true for a certain number $n=k\in\mathbb{N}$. It is necessary to prove that it remains true for $n=k+1$. By the
 condition, \smash{$Y(\cdot)$ belongs to $W_p^{k+1}$} and $Y(t)\neq\{0\}$ for any $t\in [a, b]$. Hence, by the inductive assumption,
 $Y^{-1}(\cdot)$ belongs to $W_p^{k}$. Thus, the function $$\left(Y^{-1}\right)'(\cdot)=-Y'(\cdot)Y^{-2}(\cdot)$$ belongs to the space $W_p^{k}$
 because it is a Banach algebra. Therefore, $Y^{-1}(\cdot)$ belongs to $W_p^{k+1}$.

Thus, for $m=1$ the lemma is proved.

We now prove the lemma for $m\geq 2$. It is known that
\begin{equation}\label{r2}
Y^{-1}(t) = \frac{1}{\det Y(t)} Y^T (t).
\end{equation}
Here, $Y^T (\cdot)$ is the transposed matrix function formed by the cofactors of elements of the matrix function $Y(\cdot)$. By the condition, $Y(\cdot)$ belongs to $(W_p^{k+1})^{m\times m}$. Thus, $Y^T (\cdot)$ belongs to $(W_p^{k+1})^{m\times m}$ because the functional class $W_p^{k+1}$ forms a Banach algebra. By using the result established above and equality (\ref{r2}), we conclude that $Y^{-1}(\cdot)$ also belongs to $(W_p^{k+1})^{m\times m}$.
\end{proof}

Further, we introduce a metric space of nondegenerate matrix functions
$$ \mathcal{Y}_{p}^{n}:=\{Y(\cdot)\in (W_{p}^{n})^{m\times m}\colon \,Y(a)=I_{m},\quad
  \det Y(t)\neq 0 \quad \forall t\in [a, b]\}$$
\noindent with the metric
 $$d_{n, p}(Y,Z):=\|Y(\cdot)-Z(\cdot)\|_{n,p}.$$

\begin{theorem}\label{th1} A nonlinear mapping
\begin{equation}\label{rr6}
A(\cdot)\mapsto Y(\cdot)
\end{equation}
that associates every matrix function $A(\cdot) \in (W_p^{n-1})^{m\times m}$ with a unique solution $Y(\cdot)$ of the~matrix Cauchy problem \eqref{r3} is a homeomorphism of the Banach space $(W_{p}^{n-1})^{m\times m}$ onto the~metric space $\mathcal{Y}_{p}^{n}$.
\end{theorem}

We split the proof of the theorem into three parts.

\begin{lemma}\label{l2} The nonlinear mapping \eqref{rr6} is a bijection of the space $(W_{p}^{n-1})^{m\times m}$ onto the metric space $\mathcal{Y}_{p}^{n}$.
\end{lemma}

\begin{proof} We prove the lemma by induction on the parameter $n\in\mathbb{N}$.

 First, we prove Lemma \ref{l2} for $n=1$. Let $A(\cdot)$ belongs to $(L_p)^{m\times m}$ and let $Y(\cdot)$ be a unique solution of problem \eqref{r3}. Since $Y(\cdot)$ belongs to $(L_p)^{m\times m}$ (as a continuous function) and \smash{$Y'(\cdot)=-A(\cdot)Y(\cdot)$} belongs to $(L_p)^{m\times m}$, we conclude that $Y(\cdot)$ belongs to $(W_p^1)^{m\times m}$. Since the~solution of the Cauchy problem (\ref{r3}) is unique, $Y(\cdot)\in \mathcal{Y}_{p}^{1}$ is uniquely determined by the coefficient $A(\cdot)\in (L_p)^{m\times m}$. Hence, mapping \eqref{rr6} is injective.

We now prove surjectivity of the mapping. By the Liouville--Jacobi formula (see, e.g.,~\cite{Yak}), we find
$$ \operatorname{det} Y(t)=\operatorname{det} Y(a)\operatorname{exp}\left(\int_{a}^t \operatorname{sp}A(s){\rm d}s\right)=
\operatorname{exp}\left(\int_{a}^t \operatorname{sp} A(s){\rm d}s\right)\neq 0.$$
\noindent Hence, the matrix $Y(t)$ is nondegenerate for any $t\in [a, b]$. Then there exists an inverse matrix $Y^{-1}(t)$ and equation (\ref{r3}) can be rewritten in the form:
\begin{equation}\label{r5}
A(t)=-Y'(t) Y^{-1}(t).
\end{equation}

At the same time, $Y'(\cdot)$ belongs to $(L_p)^{m\times m}$ and, by Lemma \ref{l1}, $Y^{-1}(\cdot)$ belongs to $(W_{p}^{1})^{m\times m}$. Thus, the product of these matrix function $A(\cdot)\in (L_p)^{m\times m}$, and mapping \eqref{rr6} is surjective. Therefore, this mapping is actually bijective.

 Assume that assertion of Lemma \ref{l2} is true for a certain number $n=k\in \mathbb{N}$. It is necessary to prove that it remains true for $n=k+1$. Let $$A(\cdot)\in(W_{p}^{k})^{m\times m}\subset (W_{p}^{k-1})^{m\times m}$$ and let $Y(\cdot)$ be a unique solution of problem \eqref{r3}. By the inductive assumption, $Y(\cdot)$ belongs to $\mathcal{Y}_{p}^{k}$. Hence, $$Y'(\cdot)=-A(\cdot) Y(\cdot)\in (W_{p}^{k})^{m\times m}$$ because $W_{p}^{k}$ is a Banach algebra. Thus, $Y(\cdot)$ belongs to $(W_p^{k+1})^{m\times m}$. Since the solution of the~Cauchy problem is unique, $Y(\cdot)$ belongs to $\mathcal{Y}_{p}^{k+1}$, i.e., mapping \eqref{rr6} is injective for $n=k+1$.

We now prove the surjectivity of the mapping. Since $Y(\cdot)$ belongs to $(W_{p}^{k+1})^{m\times m}$, the derivative $Y'(\cdot)$ belongs to $(W_{p}^{k})^{m\times m}$, and, by Lemma \ref{l1}, we arrive at the inclusion $$Y^{-1}(\cdot)\in (W_{p}^{k+1})^{m\times m}.$$ Thus, the product $-Y'(\cdot) Y^{-1}(\cdot)$ belongs to $(W_{p}^{k})^{m\times m}$. Since equality (\ref{r5}) is true, the matrix function $A(\cdot)$ also belongs to the Banach space~$(W_{p}^{k})^{m\times m}$. This implies that each matricant $Y(\cdot)\in \mathcal{Y}_{p}^{k+1}$ is associated, by relation~\eqref{r5}, with a matrix function $A(\cdot)\in (W_{p}^{k})^{m\times m}$.
\end{proof}

\begin{lemma}\label{l3} The solution $Y(\cdot)\in \mathcal{Y}_{p}^{n+1}$ of problem \eqref{r3} continuously depends on the coefficient \smash{$A(\cdot)\in (W_{p}^{n})^{m\times m}$}.
\end{lemma}

\begin{proof} It is necessary to show that the relation
$$\|A(\cdot;\varepsilon)-A(\cdot;0)\|_{n,p}\rightarrow 0, \quad \varepsilon\rightarrow 0+,
$$
implies that $\|Y(\cdot;\varepsilon)-Y(\cdot;0)\|_{n+1,p}\rightarrow 0$. We again proceed by induction on the parameter $n\in\mathbb{N} \cup \{0\}$.

To this end, we consider the following family of matrix problems parametrized by a number $\varepsilon\in[0,\varepsilon_0]$:
\begin{equation}\label{r6}
  Y'(t;\varepsilon)+A(t;\varepsilon)Y(t;\varepsilon)=0,\quad t\in
  (a,b),
\end{equation}
\begin{equation}\label{r7}
   Y(a;\varepsilon)=I_{m},
\end{equation}
where $$A(\cdot;\varepsilon)\in(L_p)^{m\times m}.$$

Assume that the condition
\begin{equation}\label{r8}
\|A(\cdot;\varepsilon)-A(\cdot;0)\|_{0, p}\rightarrow 0
\end{equation}
 as $\varepsilon\rightarrow 0+$ is satisfied. In this case, we can show that the uniquely defined solutions of problems
 (\ref{r6}), (\ref{r7}) satisfy the limit relation
$$
\|Y(\cdot;\varepsilon)-Y(\cdot;0)\|_{1,p}\rightarrow 0,
$$
which is equivalent to the following relation:
$$\|Y(\cdot;\varepsilon)-Y(\cdot;0)\|_{1,p}:=\|Y(\cdot;\varepsilon)-Y(\cdot;0)\|_{0, p}+
\|Y'(\cdot;\varepsilon)-Y'(\cdot;0)\|_{0, p}.$$
Hence, it suffices to show that each term on the right-hand side of this equality tends to zero.

By using condition \eqref{r8}, we get
\begin{equation*}
\|A(\cdot;\varepsilon)-A(\cdot;0)\|_{0, 1}\rightarrow 0.
\end{equation*}

In \cite{Tamar}, Tamarkin proved that this fact implies the uniform convergence of matricants
\begin{equation}\label{r12}
\|Y(\cdot;\varepsilon)-Y(\cdot;0)\|_{\infty}\rightarrow 0.
\end{equation}
Hence, $$\|Y(\cdot;\varepsilon)-Y(\cdot;0)\|_{0,p}\rightarrow 0.$$

Since the Sobolev spaces $(W_p^{n})^{m\times m}$ form a Banach algebra, relations \eqref{r8} and \eqref{r12} imply that   $$\|A(\cdot;\varepsilon)Y(\cdot;\varepsilon)-A(\cdot;0)Y(\cdot;0)\|_{0, p}\rightarrow 0.$$
Thus, by using equality (\ref{r6}), we get
$$\|Y'(\cdot;\varepsilon) - Y'(\cdot;0)\|_{0,p}\rightarrow 0.$$

 Assume that the conclusion of the lemma is true for some number $n=k\in \mathbb{N}$ and the solution $Y(\cdot)\in \mathcal{Y}_{p}^{k}$ of problem (\ref{r6}), (\ref{r7}) continuously depends on the coefficient $A(\cdot)\in (W_p^{k-1})^{m\times m}$ for $\varepsilon= 0$.

It is necessary to prove that the conclusion of the lemma remains true for $n=k+1$. Assume that the condition
$$
\|A(\cdot;\varepsilon)-A(\cdot;0)\|_{k,p}\rightarrow 0
$$
holds as $\varepsilon\rightarrow 0+$.

Further, since the Sobolev spaces form a Banach algebra, in view of the assumption made above, we conclude that
$$\|A(\cdot;\varepsilon)Y(\cdot;\varepsilon)-A(\cdot;0)Y(\cdot;0)\|_{k, p}\rightarrow 0.$$
By using equation (\ref{r6}), we get
$$\|Y'(\cdot;\varepsilon) - Y'(\cdot;0)\|_{k,p}\rightarrow 0.$$
This yields the required relation
\begin{gather*} \|Y(\cdot;\varepsilon) - Y(\cdot;0)\|_{k+1,p}\rightarrow 0.\tag*{\qed}
\end{gather*}\renewcommand{\qed}{}
\end{proof}

\begin{lemma}\label{l4} The coefficients $A(\cdot;\varepsilon)\in (W_p^{n})^{m\times m}$ for $\varepsilon= 0$
continuously depend on the solutions $$Y(\cdot;\varepsilon)\in \mathcal{Y}_{p}^{n+1}$$ of problem \eqref{r6}, \eqref{r7}.
\end{lemma}

\begin{proof} Assume that the solutions of problems (\ref{r6}), (\ref{r7}) satisfy the limit relation
\begin{equation}\label{r14}
\|Y(\cdot;\varepsilon)-Y(\cdot;0)\|_{n+1,p}\rightarrow 0
\end{equation}
as $\varepsilon\rightarrow 0+$. This enables us to prove that
 $$\|A(\cdot;\varepsilon)-A(\cdot;0)\|_{n,p}\rightarrow 0.$$

In view of assumption (\ref{r14}), we conclude that
$$\|Y'(\cdot;\varepsilon)-Y'(\cdot;0)\|_{n,p}  \rightarrow
0,$$
and, by virtue of Lemma \ref{l1}, we get
\begin{equation*}
\|Y^{-1}(\cdot;\varepsilon)-Y^{-1}(\cdot;0)\|_{n+1,p} \rightarrow
0.
\end{equation*}
By using these relations and equality (\ref{r5}), we can prove that
$$\|A(\cdot;\varepsilon)-A(\cdot;0)\|_{n,p} = \|Y'(\cdot;\varepsilon)Y^{-1}(\cdot;\varepsilon)-Y'(\cdot;0)Y^{-1}(\cdot;0)\|_{n,p}  \rightarrow
0, \quad \varepsilon\rightarrow 0+.$$

Hence, we have established the bicontinuity of the mapping
\begin{equation*} A(\cdot)\mapsto Y(\cdot)\colon (W_{p}^{n-1})^{m\times m}\rightarrow \mathcal{Y}_{p}^{n}.
 \tag*{\qed}
  \end{equation*}
  \renewcommand{\qed}{}
\end{proof}

So,  Theorem \ref{th1} is proved.

We now establish one more auxiliary statement:

\begin{lemma}\label{dija oper}
For any matrix function $Y(\cdot)\in (W^{n}_p)^{m\times m}$, a vector $q\in\mathbb{C}^m$, and a linear continuous operator $B\colon (W^{n}_p)^{m\times m} \times\mathbb{C}^m$, the following equality is true:
\begin{equation*}\label{rivn_matruz}
B(Y(\cdot)q) = \left[BY\right]q,
\end{equation*}
where the matrix $\left[BY\right]$ is given by equality \eqref{3.BY}.
\end{lemma}

\begin{proof} Assume that the matrix function $Y(\cdot)=(y_{ij}(\cdot))_{i,j=1}^m$ and the column vector \smash{$q = (q_{j})_{j=1}^m$.} We denote
\begin{equation*}(\alpha_{i})_{i=1}^m = \left[BY\right] q \quad \mbox{and}\quad (\beta_{i})_{i=1}^m = B(Y(\cdot) q).\end{equation*} Let
$$
B(y_{k}(\cdot))_{k=1}^m =: (c_{k})_{k=1}^m.
$$
As a result of the action of the operator $B$ upon the matrix function $Y(\cdot)$, we obtain a matrix
$$
\left[BY\right] = (c_{ij})_{i,j=1}^m.
$$
Thus, we get
$$
(\alpha_{i})_{i=1}^m = (c_{ij})_{i,j=1}^m (q_{j})_{j=1}^m
= \left(\sum _{j=1}^m c_{ij} q_{j} \right)_{i=1}^m.
$$
Hence, an arbitrary element $\alpha_{i}$ takes the form
$$
\alpha_{i} = \sum _{j=1}^m c_{ij} q_{j}, \quad i\in \{1,2,\ldots,m\}.
$$
However,
$$
(\beta_{i})_{i=1}^m = B\left((y_{ij}(\cdot))_{i,j=1}^m (q_{j})_{j=1}^m \right) = B \left( \sum _{j=1}^m  y_{ij}(\cdot)  q_{j} \right)_{i=1}^m=
$$
$$= \sum _{j=1}^m  \left(B y_{ij}(\cdot)\right)_{i=1}^m  q_{j} =\sum _{j=1}^m \left( c_{ij} \right)_{i=1}^m q_{j}=\left(\sum _{j=1}^m c_{ij}q_{j} \right)_{i=1}^m.
$$
This implies that $\alpha_{i} = \beta_{i}$, $i\in \{1,2,\ldots,m\}$.
\end{proof}

\section{Proofs of Theorems \ref{th_fredh-bis} and~\ref{th_invertible-bis}}\label{section5}

\begin{proof}[Proof of Theorem \ref{th_fredh-bis}]
We first prove the continuity of the operator $(L,B)$.
Since, by condition, the operator $B$ is linear and continuous, it suffices to establish the continuity of the operator~$L$, which is equivalent to its boundedness. The boundedness of the linear operator
\begin{gather*}
L\colon (W^{n}_p)^m\rightarrow (W^{n-1}_p)^m
\end{gather*}
follows from the definition of norms in the Sobolev spaces $W_p^n$ and, in addition, each of these spaces forms
a Banach algebra.

We now prove that $(L,B)$ is a Fredholm operator and find its index. We choose a fixed linear bounded
operator $C_{r,m}\colon (W^{n}_p)^m\rightarrow\mathbb{C}^{r}$. Thus, the operator $(L,B)$ admits a representation
\begin{equation*}
(L,B)=(L,C_{r,m})+(0,B-C_{r,m}),
\end{equation*}
where the operator
$$(L,C_{r,m})\colon (W^{n}_p)^m\rightarrow (W^{n-1}_p)^m\times\mathbb{C}^r,$$ and the second term is a finite-dimensional operator. By the second theorem on stability (see, e.g., \cite[Chapter~3, Section~1]{Kato_book}), the operator $(L,B)$ is Fredholm if the operator $(L,C_{r,m})$ is Fredholm and, in addition,
$$\operatorname{ind}(L,B)=\operatorname{ind}(L,C_{r,m}).$$
Hence, it suffices to show that the operator $(L,C_{r,m})$ is Fredholm and find its index by choosing a proper operator $C_{r,m}$. To this end, we consider the following three cases:

\vspace{+0.2cm}

\textbf{1.} Let $r=m$.
We set
\[C_{m,m}y:=(y_1(a),\dots , y_m(a)).\]

We find the null space and the range of this operator. Let $y(\cdot)$ belong to $\operatorname{ker}(L,C_{r,m})$. Thus, $$Ly=0 \quad \mbox{and} \quad C_{m,m}y=(y_1(a),\dots , y_m(a))=0.$$ By virtue of the theorem on the unique solvability of the Cauchy problem, we get $y(\cdot)= 0$. Hence, $$\operatorname{ker}(L,C_{m,m})=0.$$

Further, assume that $h\in (W^{n-1}_p)^m\times\mathbb{C}^m$ and $c\in\mathbb{C}^m$ are chosen arbitrarily. By Theorem~\ref{th1} there exists a vector function $y(\cdot)\in (W^{n}_p)^m$ such that
$$Ly=h, \quad (y_1(a),\dots, y_m(a))=c.$$ Hence, $$\operatorname{ran}(L,C_{r,m})=(W^{n-1}_p)^m\times\mathbb{C}^m.$$

\vspace{+0.2cm}

\textbf{2.} Let $r>m$. We set
\[C_{r,m}y:=(y_1(a),\dots, y_m(a), \underbrace{0,\dots ,0}_{r-m})\in\mathbb{C}^{r}.\]

It is necessary to determine the null space of the operator $(L,C_{r,m})$. Let $y(\cdot)$ belong to $\operatorname{ker}(L,C_{r,m})$. Then $$Ly=0 \quad \mbox{and} \quad (y_1(a),\dots , y_m(a))=0.$$ By the theorem on uniqueness of the solution of the Cauchy problem, we find $y(\cdot)= 0$.

We represent the set of values of the operator $(L,C_{r,m})$ in the form of a direct sum of two subspaces as
follows: \[\operatorname{ran}(L,C_{r,m})=\operatorname{ran}(L,C_{m,m})\oplus (\underbrace{0,\dots ,0}_{r-m}).\]
However, as shown above, $$\operatorname{ran}(L,C_{m,m})=(W^{n-1}_p)^m\times\mathbb{C}^m.$$ Hence, $$\operatorname{def} \operatorname{ran}(L,C_{r,m})=r-m.$$

\vspace{+0.2cm}

\textbf{3.} Let $r< m$. We set
\[C_{r,m}y:=(y_1(a),\dots , y_r(a))\in\mathbb{C}^{r}.\]

It is necessary to prove that
\begin{gather*}
\operatorname{dim} \operatorname{ker}(L,C_{r,m})=m-r, \\
\operatorname{def} \operatorname{ran}(L,C_{r,m})=0.
\end{gather*}
Let $y(\cdot)$ belong to $\operatorname{ker}(L,C_{r,m})$. Thus, $$Ly=0 \quad \mbox{and} \quad (y_1(a),\dots, y_r(a))=0.$$ We now consider the following $m-r$ Cauchy problems:
\begin{gather*}
Ly_k=0, \quad C_{m,m}y_k=e_k, \quad \mbox{where} \quad k\in \{r+1, r+2,\dots ,m\}, \\
e_k:=(0,\dots, 0, \underbrace{1}_{k}, 0, \dots ,0) \in {C}^{m}.
\end{gather*}
It follows from Theorem \ref{th1} that solutions of these problems are linearly independent and form a basis in the subspace $\operatorname{ker}(L,C_{r,m})$.

The surjectivity of the operator $(L,C_{r,m})$ follows from the established surjectivity of the~operator~$(L,C_{m,m})$.

Hence, in each of the analyzed three cases, the operator $(L,B)$ is a Fredholm operator with index $m-r$.
\end{proof}

\begin{proof}[Proof of Theorem~\ref{th_invertible-bis}] By virtue of Theorem \ref{th_fredh-bis}, the invertibility of the operator $(L,B)$ is equivalent to $r=m$ and $\operatorname{ker}(L,B) = \{0\}$.
Hence, it suffices to show that the condition $$\operatorname{ker}(L,B)\neq\{0\}$$ is equivalent to the singularity of the square matrix \eqref{3.BY}.

Let $\operatorname{ker}(L,B)\neq\{0\}$. Then, by Lemma \ref{dija oper}, there exists a nontrivial solution of the~homogeneous equation $(L,B)y=(0,0)$ such that
$$y(\cdot)\in \operatorname{ker}(L,B) \Leftrightarrow (\exists\:q \in\mathbb{C}^{m}\colon y(t) = Y(t)\cdot q, \, \left[BY\right]q=0), $$
where the vector $q\neq0$.
This means that the columns of matrix \eqref{3.BY} are not linearly independent and the matrix is degenerate.

Conversely, let matrix \eqref{3.BY} be degenerate. Then its columns are not linearly independent. This means that,
for some vector $q\neq0$
\begin{equation*}
[BY]q=0.
\end{equation*}
We set $y(\cdot):=Y(\cdot)q$. Then $y(\cdot)\neq0, \, Ly=0$, and
\[By=B(Y(\cdot)q) =[BY]q=0\]
by Lemma \ref{dija oper}.
Hence, $y(\cdot)\in \operatorname{Ker}(L,B)\neq\{0\}$.
\end{proof}

\bigskip

\end{document}